\documentclass[11pt,a4paper]{article}
\usepackage[active]{srcltx}
\usepackage{amsmath,amsfonts,amsthm,amssymb}
\usepackage{url, verbatim}
\usepackage{hyperref}
\usepackage[dvips]{graphicx}

\usepackage[a4paper,hcentering,vcentering]{geometry}
\newtheorem{lemma}{Lemma}[section]
\newtheorem{corollary}[lemma]{Corollary}

\newtheorem{proposition}[lemma]{Proposition}
\newtheorem{theorem}[lemma]{Theorem}
\theoremstyle{definition}

\newtheorem{remark}[lemma]{Remark}

\def\be{\begin{eqnarray}}
\def\ee{\end{eqnarray}}
\def\beal{\begin{aligned}}
\def\enal{\end{aligned}}
\newcommand{\eps}{\varepsilon}
\newcommand{\ga}{\gamma}
\newcommand{\Ga}{\Gamma}

\newcommand{\RR}{\mathbb{R}}
\newcommand{\NN}{\mathbb{N}}

\newcommand{\TT}{\mathbb{T}}

\newcommand{\PP}{\mathcal{P}}

\newcommand{\LL}{\mathcal{L}}

\newcommand{\SSS}{\mathcal{S}}

\newcommand{\OO}{\mathcal{O}}
\newcommand{\FF}{\mathcal{F}}
\newcommand{\HH}{\mathcal{H}}

\newcommand{\ii}{^{-1}}
\newcommand{\de}{\delta}
\newcommand{\pa}{\partial}

\newcommand{\al}{\alpha}

\newcommand{\ol}{\overline}
\newcommand{\La}{\Lambda}

\newcommand{\om}{\omega}

\renewcommand{\Im}{\mathrm{Im\,}}
\newcommand{\wt}{\widetilde}

\newcommand{\Res}{\mathrm{Res}}

\newcommand{\quadr}{\mathrm{quad}}
\newcommand{\Kep}{\mathrm{Kep}}
\newcommand{\per}{\mathrm{per}}
\newcommand{\oct}{\mathrm{oct}}
\newcommand{\secu}{\mathrm{sec}}

\title{Secular instability\\
  in the spatial three-body problem}
\author{J. F{\'e}joz,\thanks{Universit{\'e} Paris-Dauphine (Ceremade) \&
    Observatoire de Paris (IMCCE), France,
    \url{jacques.fejoz@dauphine.fr}.} \enspace
  M. Guardia\thanks{Universitat Polit{\`e}cnica de Catalunya (Departament
    de Matem{\`a}tica Aplicada I), Spain, \url{marcel.guardia@upc.edu}.}}

\begin{document}
\maketitle

\emph{Abstract. } Consider the spatial three-body problem, in the
regime where one body revolves far away around the other two, in
space, the masses of the bodies being arbitrary but fixed; in this
regime, there are no resonances in mean motions. The so-called secular
dynamics governs the slow evolution of the Keplerian ellipses. We show
that it contains a horseshoe and all the chaotic dynamics which goes
along with it, corresponding to motions along which the eccentricity
of the inner ellipse undergoes large, random excursions. The proof
goes through the surprisingly explicit computation of the homoclinic
solution of the first order secular system, its complex singularities
and the Melnikov potential.


\section{Introduction}%
\label{sec:intro}%

The question of the stability of the Solar System is the oldest open
problem in Dynamical Systems. A number of works have shown striking
instability mechanisms in the three-body problems, e.g. oscillatory
orbits and close to parabolic motion~\cite{Alekseev:1969,
  Sitnikov:1960, SimoL80, LlibreS80, Guardia:2013:oscillatory,
  DelshamsKRS12}, chaotic dynamics near double or triple
collisions~\cite{Bolotin:2006, Moeckel:1989:chaotic} (see
also~\cite{Moser:1973}). But scarce mathematical mechanisms have been
described regarding more astronomical regimes, which would be
plausible for subsystems of solar or extra-solar systems. One of
them~\cite{Fejoz:2015:diffusion} shows
the existence of global
instabilities along mean motion (i.e. Keplerian) resonances.

In this paper, we focus on secular resonances. Numerical evidence has
long been suggesting that such resonances are a major source of chaos
in the Solar system~\cite{Laskar:1993:chaotic, Laskar:2008:chaotic,
  Froeschle:1989}. For example, astronomers have established that
Mercury's eccentricity is chaotic and can increase so much that
collisions with Venus or the Sun become possible, as a result from an
intricate network of secular resonances~\cite{Boue:2012:simple}. On
the other hand, that Uranus's obliquity ($97^o$) is essentially
stable, is explained, to a large extent, by the absence of any
low-order secular resonance~\cite{Boue:2010:collisionless,
  Laskar:1993:chaotic}. It is the goal of this paper to provide a
simple instability mechanism in the secular dynamics of the three-body
problem.

\section{The secular Hamiltonian in the lunar regime}
\label{sec:SecularLunar}%

Consider three point masses undergoing the Newtonian attraction in
space. Call $m_0$, $m_1$ and $m_2$ the masses and $(q_j,p_j)_{j=1,2,3}
\in (\RR^3 \times \RR^3)^3$ the Jacobi coordinates. After having
symplectically reduced the system by the symmetry of translations, the
Hamiltonian depends only on $(q_j,p_j)_{j=1,2}$ and equals
$$
F_{\Kep} + F_{\per}, \quad
\begin{cases}
  F_{\Kep} =  \sum_{j=1,2} \left( \frac{p_j^2}{2\mu_j} -
    \frac{\mu_jM_j}{\|q_j\|} \right)\\
  F_{\per} = \left( \frac{\mu_2M_2}{\|q_2\|} -
  \frac{m_0m_2}{\|q_2+\sigma_1q_1\|} -
  \frac{m_1m_2}{\|q_2-\sigma_0q_1\|}\right),
\end{cases}
$$ where the reduced masses are defined by
\begin{equation}
  \label{eq:masses}
  M_0 = m_0, \quad M_1 = m_0 + m_1, \quad \mu_1^{-1} =
  M_0^{-1} + m_1^{-1}, \quad \mu_2^{-1} = M_1^{-1} +
  m_2^{-1}
\end{equation}
and coefficients $\sigma_j$ are the barycentric weights of the inner
bodies: $\sigma_j = m_j/M_1$ ($j=0,1$). $F_{\Kep}$ is the integrable
Hamiltonian of two uncoupled Kepler problems and induces a
``Keplerian'' action of the $2$-torus, while $F_{\per}$ is the
so-called perturbing function.

In this paper, we consider the asymptotic regime (sometimes called
\emph{lunar} or \emph{hierarchical} or \emph{stellar}) where the
masses are fixed, while body $2$ revolves far away around the other
two. In particular, each of the two terms of $F_{\Kep}$ is negative,
the outer eccentricity is bounded away from $1$, and the semi major
axes $a_1$ of $q_1$ and $a_2$ of $q_2$ satisfy
$$a_1 = O(1) \ll a_2 \to \infty.$$ 
The expansion of $F_{\per}$ in the powers of $\|q_1\|/\|q_2\| \ll 1$
is
\begin{equation}
  \label{eq:Fexpansion}%
  F_{\per} = -\frac{\mu_1m_2}{\|q_2\|} \sum_{n\geq 2} \sigma_n
  P_n(\cos \zeta) \left( \frac{\|q_1\|}{\|q_2\|} \right)^n,
\end{equation}
where the $P_n$'s are the Legendre polynomials, and $\sigma_n =
\sigma_0^{n-1} + (-1)^n \sigma_1^{n-1}$, $n\geq 2$. Note that the 
perturbing function is of the order of $1/a_2^3$.

Standard normal form theory\footnote{See, in this context,
  \cite{Fejoz:2002:quasiperiodic}: mean motions do not interfere since
  they are not of the same order.} shows that, for every integer $k$,
there is a local change of coordinates $a_2^{-3}$-close to the
identity which transforms the Hamiltonian into
\begin{equation}
  \label{eq:F}
  F=F_\Kep+F_\secu+O(a_2^{-k-1}),
\end{equation}
where $F_{\secu}$ is the \emph{secular
  Hamiltonian}\footnote{\emph{Secular}, means century in Latin. This
  term governs the slow dynamics (as opposed to the fast, Keperian
  dynamics), on a long time scale, symbolically one century.}  of
order $k$
\begin{equation}
  \label{eq:Fsec}
  F_{\secu} = \int_{\TT^2} F_{\per} \, d\ell_1 \, d\ell_2 + O \left(
    \frac{1}{a_2} \right),
\end{equation}
obtained by averaging out the mean anomalies $\ell_j$ at successive
orders; $\TT^2$ is the orbit of the Keplerian action defined by
$F_{\Kep}$, parameterized by the mean anomalies $\ell_1$ and $\ell_2$
of the two planets. The secular Hamiltonian descends to the quotient
by the Keplerian action of $\TT^2$, and induces a Hamiltonian system
on the space of pairs of Keplerian ellipses with fixed semi major
axes, describing the slow evolution of the Keplerian ellipses due to
the perturbation. 

In this paper, we primarily consider the principal part of the
Hamiltonian $F$ of equation~\eqref{eq:F}, i.e.  $F_{\Kep} +
F_{\secu}$.

We establish the phase portrait of the secular system, in particular
with a (well known) hyperbolic periodic orbit located at inclined and
nearly circular ellipses. The main results of the paper are the
following. We prove that this periodic orbit posesses transversal
homoclinic points (Corollary \ref{prop:Splitting}) and that the
secular system posesses a horseshoe (Theorem \ref{thm:horseshoe}), and
therefore chaotic motions. As a by-product, we prove that the secular
system is non-integrable. The results we obtain are valid for any
value of the masses of the three bodies.

To prove these results, we take advantage of the fact that the first
order in $1/a_2$ of the secular Hamiltonian is integrable. We
explicitely determine the homoclinic solution to the hyperbolic
periodic orbit of this first order (Lemma \ref{lm:separatrix}).  Then,
we show that the Melnikov potential associated with this homoclinic
orbit has non-degenerate critical points (Proposition
\ref{prop:Melnikov}), from what it follows that the secular homoclinic
solution splits for the full secular system, as well as for the full
initial system.

The result we present in this paper is a step towards proving the
existence of unstable orbits in the three body problem in the lunar
regime, for any value of the masses of the bodies, in the sense of
Arnold diffusion. By ``unstable orbits'' here we mean orbits such that
one of the three bodies undergoes a large change in the semi major
axis of the associated osculating ellipse. This is explained more
precisely in Section \ref{sec:Diffusion}.

\subsection{The quadrupolar and octupolar Hamiltonians}
The first two terms of the expansion~\eqref{eq:Fexpansion} of
$F_{\per}$, after averaging, yield the \emph{quadrupolar} and
\emph{octupolar} Hamiltonians:\footnote{More generally, the $n$-th
  order term is called $2^n$-polar, because, in electrostatic, this
  term is the dominating term of the potential of a system of $2^n$
  well chosen charged particles.}
\begin{equation}
  \label{eq:Fav}
  F_{\secu} = -\mu_1m_2 \left(F_{\quadr} + (\sigma_0-\sigma_1) F_{\oct}
  \right) + O \left( \frac{a_1^4}{a_2^5} \right),
\end{equation}
with
\begin{equation}
  \label{eq:Fquad0}
  F_{\quadr} = \frac{{a_1}^2}{8\,{a_2}^3\,\left(1-{e_2}^2\right)^{3/2}}
  \,\left( \left(15\,{e_1}^2\,\cos ^2 {g_1} -12\,{e_1}^2 -3
    \right)\,\sin ^2i+3\, {e_1}^2+2\right)
\end{equation}
and
\begin{equation}
  \label{eq:Foct}
  F_{\oct}
  =- \frac{15}{64} \frac{a_1^3}{a_2^4} \frac{e_1e_2}{(1-e_2^2)^{5/2}}
  \times
\end{equation}
\begin{equation*}
\left\{
  \begin{split}&
    \cos {g_1}\cos {g_2}\left[
      \begin{split}&
        \frac{G_1^2}{L_1^2} \left(5\sin^2i(-7 \cos^2g_1+6) -3\right)\\&
      -35 \sin^2g_1 \sin^2 i +7
      \end{split}
    \right]\\&
    - \sin {g_1}\sin {g_2}\cos i\left[
      \begin{split}&
        \frac{G_1^2}{L_1^2} \left(
          5\sin^2i(7\cos^2g_1-4)
          +3
        \right)\\&
        +35\sin^2{g_1}\sin^2i-7
      \end{split}
    \right]
  \end{split}
\right\}.
\end{equation*} 
this is a standard computation (see~\cite{Fejoz:2002:quasiperiodic}
for the reduction to integrating trigonometric polynomials), and we
have used the following notations:
$$\begin{cases}
  e_j &=\mbox{eccentricities}\\
  g_j &=\mbox{arguments of pericenters}\\
  i &=\mbox{mutual inclination}.
\end{cases}$$ 
The Hamiltonians $F_{\quadr}$ and $F_{\oct}$ do not depend on the order $k$
of averaging in the expression~\eqref{eq:F} (i.e., due to the special
dependence of $F_{\Kep}$ and $F_{\per}$ in $L_2$, these first two terms
of the normal form are not modified by second and higher order
averaging). 

Recall that we want to analyze the secular Hamiltonian $F_\secu$ for
any value of the masses, so that parameters $m_i$, $\mu_i$ and $M_i$
are just given constants satisfying~\eqref{eq:masses}.


Let $(\ell_j,L_j,g_j,G_j,\theta_j,\Theta_j)_{j=1,2}$ denote the
Delaunay variables. Jacobi's classical elimination of the node
consists in considering a codimension-$3$ submanifold of fixed,
vertical angular momentum, and quotienting by horizontal
rotations. The reduced manifold has dimension $8$, on which the
Keplerian and eccentric variables $(\ell_j,L_j,g_j,G_j)_{j=1,2}$
induce symplectic coordinates. After averaging out the mean motions we
are left with the symplectic coordinates $(g_j,G_j)_{j=1,2}$, and the
variables $L_j$ may be treated as parameters.

The variables $L_i$ are given by $L_i=\mu_i \sqrt{M_ia_i}$. Therefore, in the
lunar regime we have $L_1\sim 1$ and $L_2\gg 1$. Recall that the eccentricity
of the ellipse is given in terms of the Delaunay coordinates by
\begin{equation}\label{def:eccentricity}
 e_i=\sqrt{1-\frac{G_i^2}{L_i^2}}.
\end{equation}
Since we want the outer body to describe non-degenerate ellipses, we
even assume $G_2\sim L_2$. Since we are doing a perturbative analysis
in $L_2^{-1}$, we define the new variable
\[
 \Ga=C-G_2,\,\, \text{where }C=\de L_2\text{ and }\de>0 \text{ is a fixed
constant}.
\]
The coordinates
\begin{equation}\label{def:CoordinatesInGamma}
(g_1,G_1,\gamma,\Gamma) = (g_1,G_1,-g_2,C-G_2), 
\end{equation}
are symplectic; we  also call them Delaunay coordinates (as
opposed to some radically different coordinates used later). Now these
variables are bounded as $a_2, L_2 \to \infty$ (recall that $C \sim
G_2 \to \infty$). At the first order in $L_2^{-1}$, the mutual
inclination $i$ satisfies
\begin{equation}
  \label{eq:cosi}
  \cos^2 i = \frac{(C^2 - G_1^2 - G_2^2)^2}{4G_1^2G_2^2} =
  \frac{(C^2 - G_1^2 - (C-\Gamma)^2)^2}{4G_1^2 (C-\Gamma)^2} =
  \frac{\Gamma^2}{G_1^2}+O\left(L_2^{-1}\right).
\end{equation}
Now we express the secular Hamiltonian in these new variables and
expand it in inverse powers of $L_2$.  Recall that variables $L_j$ are
parameters, as well as the norm $C$ of the angular momentum.

\begin{lemma}\label{lemma:SecularExpansion}
The secular Hamiltonian \eqref{eq:Fav} has the form
\begin{equation}\label{def:HSecExpanded}
\begin{split}
  F_{\secu} = \al_0+\al_1 L_2^{-6}&\Big( H_0 + L_2^{-1} H_1 +
  (\sigma_0-\sigma_1) L_2^{-2} H_2 +\OO(L_2^{-3}) \Big). 
\end{split}
\end{equation}
with
\begin{equation}\label{def:FirstOrderQuadrupolar}
  H_0=\left( 1 - \frac{G_1^2}{L_1^2} \right)
  \left[ 2 - 5 \left( 1 - \frac{\Gamma^2}{G_1^2} \right) \sin^2 g_1 
  \right] - \frac{\Gamma^2}{L_1^2}\\
\end{equation}
\end{lemma}
In the coordinates $(g_1,G_1,\gamma,\Gamma)$, $F_{\quadr}$ and hence
$H_0$ and $H_1$ do not depend on $\gamma$. We do not compute $H_1$
explicitly, here. The reason is that $H_1$ does not break the
integrability of $H_0$ and therefore does not play any role in the
Melnikov analysis. In contrast, the Hamiltonian $H_2$ does break
integrability (as it will follow from our study), and is computed in
Section~\ref{sec:Melnikov}.

\begin{proof}[Proof of Lemma \ref{lemma:SecularExpansion}]
Formula \eqref{def:eccentricity} implies
\[\frac{1}{G_1^2}  = \frac{e_1^2}{G_1^2}+\frac{1}{L_1^2}.\]
Hence, using \eqref{eq:cosi},
\begin{equation}
  \label{eq:sini}
  \sin^2 i = 1 - \frac{\Gamma^2}{G_1^2}+O\left(L_2^{-1}\right)= 1-\Gamma^2 
  \frac{e_1^2}{G_1^2}-\frac{\Gamma^2}{L_1^2}+O\left(L_2^{-1}\right).
\end{equation}
Besides, 
\[e_2 = \sqrt{1- \frac{G_2^2}{L_2^2}}= \sqrt{ 1 -
  \frac{C^2}{L_1^2}}+O\left(L_2^{-1}\right)\]
is a first integral of the first order approximation of $F_{\quadr}$. 
Hence, expanding in powers of $L_2^{-1}$, $F_{\quadr}$ can be written as 
\[
 F_{\quadr}(g_1,G_1,\Gamma)=\al_1 L_2^{-6} (H_0(g_1,G_1,\Gamma)+2)+O(L_2^{-1})
\]
with
\[
 \al_1=-\frac{L_1^4
M_2\mu^6 m_2}{8M_1^2\mu_1^3(1-e_2^2)^{3/2}}
\]
and 
\begin{align*}
  H_0 &= \left.\left(5e_1^2\cos^2 g_1 - 4 e_1^2 - 1\right) \sin^2
  i +
  e_1^2\right|_{L_2^{-1}=0} \quad \mbox{(using \eqref{eq:Fquad0})} \\
  &= e_1^2 \left[ \left( 5 \cos^2 g_1 - 4 \right) \left( 1 -
      \frac{\Gamma^2}{G_1^2} \right) + \frac{\Gamma^2}{G_1^2} + 1
  \right] - \frac{\Gamma^2}{L_1^2}
  \quad \mbox{(using \eqref{eq:cosi} and \eqref{eq:sini})}\\
  &= e_1^2 \left[ 2 - 5 \left( 1 - \frac{\Gamma^2}{G_1^2}
    \right) \sin^2 g_1 \right] - \frac{\Gamma^2}{L_1^2}.
\end{align*}
Factorization~\eqref{def:FirstOrderQuadrupolar} follows. 
\end{proof}

\section{The quadrupolar phase portrait}
\label{sec:Fquad}%

According to~\eqref{eq:Fquad0}, the quadrupolar Hamiltonian
$F_{\quadr}$ and thus $H_0$ do not to depend on $\gamma$. This happy
coincidence (which does not repeat itself for the next order term
$F_{\oct}$) makes $F_{\quadr}$ integrable. This has been extensively
used (see~\cite{Zhao:2014:collisions} and references therein). Here,
we may thus think of $\Gamma$ as a parameter.

Complex singularities of solutions of $F_{\quadr}$ are hard to
determine in general. In our regime, the first order of the quadrupolar 
Hamiltonian, $H_0$, can
be factorized as described in
equation~\eqref{def:FirstOrderQuadrupolar} (up to the additive
constant $-\Gamma^2/L_1^2$), which  dramatically simplifies our
study.


The Hamiltonian \(H_0\) is analytic on a neighborhood of the
cylinder \[(g_1,G_1)\in \TT^1 \times \, ]0,L_1[\] in \(\TT^1 \times
\RR\). Since it is $\pi$-periodic with respect to $g_1$, we may focus
on the domain $0 \leq g_1 \leq \pi$, $0 \leq G_1 \leq L_1$, keeping in
mind that the Delaunay coordinates blow up circular ellipses ($G_1 =
L_1$).

Hamilton's vector field is
\[
\begin{cases}
  \dot g_1
  = \frac{2G_1}{L_1^2} \left[ 5 \left( 1 - \frac{\Gamma^2}{G_1^2}
    \right) \sin^2 g_1 - 2 \right] - 10 \left( 1 - \frac{G_1^2}{L_1^2}
  \right) \frac{\Gamma^2}{G_1^3} \sin^2 g_1\\  
  \dot G_1 =  5\left(1-\frac{G_1^2}{L_1^2}\right)
  \left(1-\frac{\Gamma^2}{G_1^2}\right)\sin 2g_1.
  \end{cases}
\]
The second component $\dot G_1$ vanishes if and only if (assuming $G_1
>0$)
\[G_1 \in \{|\Gamma|, L_1\} \quad \mbox{or} \quad g_1 = 0
\pmod{\pi/2}.\] If $G_1 = |\Gamma|$, $\dot g_1$ cannot
vanish. Moreover, $G_1 \in \, ]0,L_1[\,$. Let
\begin{equation}
  \label{eq:Gammat}%
  \tilde\Gamma = \sqrt{ 1 - \frac{\Gamma^2}{L_1^2} } \in \, ]0,1[.
\end{equation}
The equilibrium points thus satisfy the following equations:
\begin{itemize}
\item If $G_1=L_1$ (circular ellipse),
  \begin{equation}
    \label{eq:hyperbolic}%
    \sin^2 g_1 = \frac{2}{5 \tilde\Gamma^2}.
  \end{equation}
  Assuming that
  \begin{equation}
    \label{cond:beta-upperbound}%
    |\Gamma| < L_1\sqrt{\frac{3}{5}},
  \end{equation}
  whence \(\tilde\Gamma^2 \geq \frac{2}{5}\) (this condition is
  satisfied if the inner ellipse has a large eccentricity or an
  inclination close to $\pi/2$), within the $g_1$-interval $[0,\pi[$,
  there are two solutions $g_1^{\min} \in ]0,\pi/2[$ and $g_1^{\max}$
  symmetric with respect to $\pi/2$, which are located on the
  energy level $H_0=-\Ga^2/L_1^2$.

  In variables $(x,y)$ such that
  \[\sin^2 g_1 = \frac{2}{5 \tilde\Gamma^2} (1+x) \quad \mbox{and} \quad
  G_1 = L_1(1-y)\] (which are local coordinates in the neighborhood of
  either one of the two above singularities), we have
  \[H_0+\Ga^2/L_1^2= 4xy + O_3(x,y);\] 
  thus the two singularities are hyperbolic. 
\item If $g_1 = 0 \pmod{\pi}$, $G_1=0$ (collision ellipse) is an
  equilibrium point of the regularized Hamiltonian vector field $G_1^3
  X_{F_{\quadr}^0}$. 
\item If $g_1=\pi/2 \pmod{\pi}$, 
  \[3G_1^4 - 10 \Gamma^2 G_1^2 + 5 L_1^2\Gamma^2 = 0\]
  (two pairs of opposite real solutions).
\end{itemize}

\begin{figure}[h]
  \centering
  \includegraphics[scale=0.7]{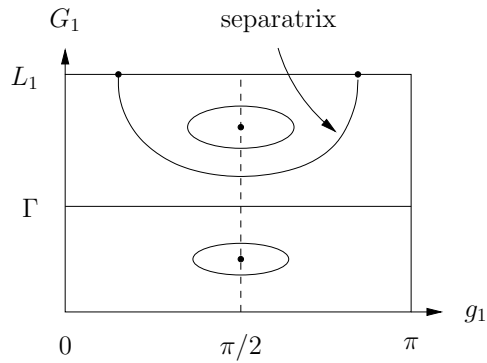}
  \caption{Phase portrait of the reduced quadrupolar dynamics}
  \label{fig:portrait}
\end{figure}

We  focus on the first two hyperbolic singularities, with
$G_1=L_1$, henceforth assuming that
condition~\eqref{cond:beta-upperbound} is satisfied, with, say,
\begin{equation}\label{def:GammaPos}
\Gamma >0.
\end{equation}

Lifted to the full secular phase space, these critical points become
the periodic orbits
\begin{equation}\label{def:PeriodicH0}
Z^0_{\min,\max}(t, \ga_0)= \left(g_1^{\min,\max}, L_1,
\ga^0+\ga^1(t),\Ga\right),
\end{equation}
where $\ga^0$ is just the initial condition for the variable $\ga$ and 
\begin{equation}\label{def:gamma1}
\ga^1(t)= -2\frac{\Ga}{L_1^2}t.
\end{equation}
The whole circle $G_1=L_1$ corresponds to circular motion. Along this
circle the Delaunay variables are singular. Thus, in a neighborhood of
circular motion it is more reliable to use the Poincar{\'e} variables
\begin{equation}\label{def:PoincareVariables}
\begin{cases}
 \xi&=\sqrt{2(L_1-G_1)}\cos g_1\\
 \eta&=-\sqrt{2(L_1-G_1)}\sin g_1.
\end{cases}
\end{equation}
Those secular coordinates are symplectic (a well known fact). The
Poincar{\'e} variables transform the Hamiltonian $H_0$ into
\begin{equation}\label{def:H_0Poincare}
 \wt H_0(\xi,\eta)= - \frac{\Gamma^2}{L_1^2}+\frac{1}{L_1}\left[2\xi^2-\left(
3-5\frac{\Ga^2}{L_1^2}\right)\eta^2\right]+\OO_2\left(\xi^2+\eta^2\right)
\end{equation}
and the line segment $\{G_1=L_1\}$ blows down to a single hyperbolic
periodic orbit
\[(\xi,\eta,\ga,\Ga)=\left(0,0,\ga^0+\ga^1(t),\Ga\right)\]
in the secular
space. 

As it has occured above, in the $(g_1,G_1)$-plane there are two
hyperbolic fixed points for $g_1\in [0,\pi]$ and two more for $g_1\in
[\pi,2\pi]$. On the line $G_1=L_1$, there are heteroclinic connections
between them. All these objects blow down to the critical point
$(\xi,\eta)=(0,0)$ in Poincar{\'e} variables.  Moreover, there are two
separatrix connections in the region $G_1<L_1$, one between the
critical points with $g_1\in [0,\pi]$ and the other one between the
critical points with $g_1\in [\pi,2\pi]$. Even if $H_0$ is well
defined for $G_1>L_1$, the corresponding solutions are spurious. In
the secular space, we obtain a hyperbolic critical point at
$(\xi,\eta)=(0,0)$ with two homoclinic orbits forming a figure eight.

The main technical goal of this paper is to show that these
separatrices split when one considers the complete secular Hamiltonian
of arbitrary order $k$. We focus on the homoclinic on which $g_1\in
[0,\pi]$. We use both Delaunay and Poincar{\'e} variables, but need to
keep in mind that the Delaunay variables are not defined on the
secular space proper along circular inner ellipses ($G_1=L_1$). Note
the same orbit is homoclinic for the secular Hamiltonian
\eqref{def:H_0Poincare} and heteroclinic after blow up (see
\eqref{def:FirstOrderQuadrupolar}).

The expression of the energy~\eqref{def:FirstOrderQuadrupolar} yields
the following characterization of this orbit, where the interval \(]
g_1^{\min}, g_1^{\max}[ \subset \, ]0, \pi [\) is defined by the
inequality (recall~\eqref{eq:Gammat})
\begin{equation}
  \label{eq:g1-interval}
  \sin^2 g_1 > \frac{2}{5\tilde\Gamma^2}
\end{equation}
and
\begin{equation}
  \chi = \sqrt{\frac{2}{3}}
  \frac{|\Gamma|}{L_1}\frac{1}{\sqrt{1-\frac{5}{3}
      \frac{\Gamma^2}{L_1^2}}} >0
  \quad \mbox{and} \quad 
  A_2 =  30 \sqrt{\frac{10}{3}}
  \frac{|\Gamma|^3}{L_1} \sqrt{1 - \frac{5}{3}
    \frac{\Gamma^2}{L_1^2}}.
   \label{eq:A2-chi}%
\end{equation}

\begin{lemma}
  \label{lm:separatrix}%
  Assume conditions
  \eqref{cond:beta-upperbound}--\eqref{def:GammaPos}.  The Hamiltonian
  $H_0$ given in \eqref{def:FirstOrderQuadrupolar} has a heteroclinic
  solution which tends to the periodic orbits $Z_{\min}^0$ and
  $Z_{\max}^0$ in \eqref{def:PeriodicH0} in backward and forward times
  respectively. Its orbit is defined by the equation
  \begin{equation}
    \label{eq:separatrix}
    \left(1-\frac{\Gamma^2}{G_1^2}\right)\sin^2 g_1=\frac{2}{5},
  \end{equation}
  with range \(g_1 \in \, ] g_1^{\min}, g_1^{\max}[ \subset \, ]0, \pi
  [\).
Its time parameterization is given by
\begin{equation}\label{def:separatrix}
Z^0(t,\gamma^0)=( 
g_1^h(t),G_1^h(t),\gamma^h(t),\Gamma)
\end{equation}
with
\begin{align}
  \cos g_1^h(t) &= - \sqrt{\frac{3}{5}} \frac{\sinh A_2t}{\sqrt{\chi^2
      + (1+\chi^2) \sinh^2 A_2t}}  \label{eq:cosg1-t}\\
  G_1^h(t) &= |\Gamma| \sqrt{\frac{5}{3}} \sqrt{1 + \frac{3}{5}
    \frac{L_1^2}{\Gamma^2} \sinh^2 A_2t} \, \left(\cosh A_2t
  \right)^{-1}\label{eq:G2}\\
  \gamma^h(t)&=  \gamma^0 + \gamma^1(t) +\gamma^2(t),
  \quad
  \gamma^2(t) = \arctan \left( \chi^{-1} \tanh A_2t \right),
  \label{eq:gamma-t}%
\end{align}
where $\gamma^0\in\TT$ is the arbitrary value  of the
$\gamma$ coordinate at initial time, $\gamma^1$ has been introduced in
\eqref{def:gamma1}.
\end{lemma}
In the $\gamma$-component of the separatrix, the angle $-\gamma^0$ is
the (arbitrary) argument of the outer pericenter at the symmetry
center of the separatrix $g_1 = \pi/2$ (recall the change of coordinates 
\eqref{def:CoordinatesInGamma}). The term $\gamma^1(t)$ is the
rotational part and the term $\gamma^2(t)$ is the transient part of
$\gamma^h(t)$.

\begin{remark}\label{rmk:HomoclinicLevelEnergy}
 The assumed condition \eqref{cond:beta-upperbound} implies that $\pa_\Ga
H_0\neq 0$ in a neighborhood of the heteroclinic orbit. Therefore, one can
rephrase Lemma \ref{lm:separatrix} as the
existence of a heteroclinic orbit in each energy level $H_0(g_1,G_1,\ga,\Ga)=h$
for  any $h<0$ (recall that by
the expression of $H_0$ in \eqref{def:FirstOrderQuadrupolar} the heteroclinic
orbits always have negative energy).
\end{remark}



\begin{proof}
On the separatrix, we have
\[\dot g_1= 10\Gamma^2\frac{1-G_1^2/L_1^2}{G_1^3}\sin g_1^2.\]
Using~\eqref{eq:separatrix}
and~\eqref{eq:g1-interval}, one can eliminate $G_1$ by
\begin{equation}
  \label{eq:G1-g1}%
  G_1 = \frac{|\Gamma|\sqrt{5} \sin g_1}{\sqrt{3 - 5\cos^2 g_1}},
\end{equation}
and get a closed differential equation
\begin{equation}
  \label{eq:g1}
  \dot g_1 = A_1 \frac{\left( 1-\frac{5}{3}(1+\chi^2) \cos^2
      g_1 \right)\sqrt{1-\frac{5}{3} \cos^2 g_1}}{\sin g_1}
\end{equation}
where $\chi$ has been defined in \eqref{eq:A2-chi} and 
\[ A_1 = 30 \sqrt{3} \Gamma^2 \left(1 -
  \frac{5}{3}\frac{\Gamma^2}{L_1^2} \right) >0. \] Note that, due to
$\Gamma\neq 0$ and~\eqref{eq:g1-interval}, $\dot g_1 \neq 0$ along the
separatrix. So, one can separate variables and, choosing $g_1 = \pi/2$
at $t=0$ yields
\[A_1t = \int_{\pi/2}^{g_1} \frac{\sin g_1 \, dg_1}{\left( 1-
    \frac{5}{3}(1+\chi^2) \cos^2 g_1 \right) \sqrt{1 - \frac{5}{3}
    \cos^2 g_1}}.\] The variable $\cos g_1 \in \, \left]-
  \frac{\sqrt{3}}{\sqrt{5(1+\chi^2)}},
  \frac{\sqrt{3}}{\sqrt{5(1+\chi^2)}} \right[$ is a coordinate on the
separatrix. Let $x$ be defined by
\[\sqrt{\frac{5}{3}} \cos g_1 = \cos x \in \left]0, \frac{1}{(1+\chi^2)}
\right[.\]  
Then 
\begin{align*}
  A_1t &= \sqrt{\frac{3}{5}} \int_{\pi/2}^{\arccos \sqrt{\frac{5}{3}}
    \cos g_1} \frac{dx}{1 - (1+\chi^2) \cos^2 x}\\
  &=\frac{1}{2\chi} \sqrt{\frac{3}{5}} \ln \frac{\tan x - \chi}{\tan x
    + \chi},
    \quad\text{with }\, \cos x = \sqrt{\frac{5}{3}} \cos g_1.
\end{align*}
One can check that the constant $A_2$ defined in \eqref{eq:A2-chi}
satisfies $A_2 = \sqrt{\frac{5}{3}} A_1 \chi$. The equation can be
solved for $\tan x$, which yields $\tan x = - \chi \coth (A_2t)$. This
equality, jointly with
\begin{align*}
  \cos g_1 &= \sqrt{\frac{3}{5}} \cos x= \sqrt{\frac{3}{5}} \frac{\pm
1}{\sqrt{1+ \tan^2 x}},
\end{align*}
gives formula \eqref{eq:cosg1-t}.
Using~\eqref{eq:G1-g1},  the $G_1$-component of the separatrix given in
\eqref{eq:G2} can be easily obtained.
Finally, only the $\gamma$-component remains to be computed. But,
along the separatrix solution, $\dot \gamma = \partial_\Gamma
H_0$. Using~\eqref{def:FirstOrderQuadrupolar}, \eqref{eq:cosg1-t}
and~\eqref{eq:G1-g1},
\begin{align*}
\dot \gamma =& 10 \frac{\Gamma}{G_1^2}
\left(1-\frac{G_1^2}{L_1^2}\right)\sin^2 g_1 - 2\frac{\Gamma}{L_1^2}\\
 =& \frac{4\Gamma}{L_1^2\chi^2} \left( 1 - \frac{5}{3}
  (1+\chi^2) \cos^2 g_1 \right) - 2\frac{\Gamma}{L_1^2} \\
  =&
\frac{4\Gamma}{L_1^2} \frac{1}{\left( (1+\chi^2) \cosh^2
    A_2t - 1 \right)} - 2\frac{\Gamma}{L_1^2}.
\end{align*}
Since $\frac{4\Gamma}{L_1^2\chi A_2}
=\frac{4\Gamma\sqrt{3}}{L_1^2\chi^2A\sqrt{5}}=1$,
equation~\eqref{eq:gamma-t} follows.
\end{proof}

\section{Splitting of separatrices}

Lemma \ref{lm:separatrix} shows that the Hamiltonian $H_0$ in
\eqref{eq:Fquad0} has a heteroclinic connection (homoclinic  in the 
secular space for the 
Hamiltonian \eqref{def:H_0Poincare}).

The next term in the asymptotic expansion of the secular Hamiltonian
given in Lemma \ref{lemma:SecularExpansion} does not depend on $\ga$
and therefore does not break integrability. For the amended
Hamiltonian, according to classical perturbation theory, the
hyperbolic periodic orbit and the homoclinic orbit persist. Moreover,
the periodic orbit remains at $(\xi,\eta)=(0,0)$. Thus, in Delaunay
coordinates, there are two hyperbolic periodic orbits, which are of
the form
\[
\begin{cases}
Z_{\min}^\eps(t,\gamma^0)&= ( g_{\min}+O(\eps),L_1
,\ga^0+\gamma^{1}(t)+O(\eps),\Gamma)\\
Z_{\max}^\eps(t,\gamma^0)&= ( g_{\max}+O(\eps),L_1
,\ga^0+\gamma^{1}(t)+O(\eps),\Gamma),
\end{cases}
\]
where $\eps=L_2\ii$. We can choose the time origin so that 
\[
 Z_{\min,\max}^\eps(0,\gamma^0)= ( g_{\min,\max}+O(\eps),L_1
,\gamma^0,\Gamma).
\]

The second order $H_2$ in Lemma \ref{lemma:SecularExpansion} depends
on $\ga$ and therefore may (and will) break integrability. If we
consider this Hamiltonian in Poincar{\'e} coordinates
\eqref{def:PoincareVariables}, it posesses a hyperbolic critical orbit
$O(L_2^{-2})$-close to $(\xi,\eta)=(0,0)$.

To show that $H_2$ makes the stable and unstable invariant manifolds
of this periodic orbit split, we use the classical Poincar{\'e}-Melnikov
Theory~\cite{Melnikov63}. More precisely, here we define the
Poincar{\'e}-Melnikov potential
\[
\begin{split}
  \LL^\eps(\ga^0)=&\int_{0}^{+\infty}
  H_2\circ Z^\eps(t,\gamma^0)-H_2\circ Z_{\max}^\eps(t,\gamma^0)\,dt\\
  &+\int_{-\infty}^{0}
  H_2\circ Z^\eps(t,\gamma^0)-H_2\circ Z_{\min}^\eps(t,\gamma^0)\,dt,
\end{split}
\]
similarly to \cite{DelshamsG00}), but taking into account the fact
that the first order perturbation does not break integrability.

In other words, we consider as an integrable system $\HH_0=H_0+\eps
H_1$ and as a perturbation $\HH_1=H_2$ (see Lemma
\ref{lemma:SecularExpansion} and recall that we have taken
$\eps=L_2\ii$). In Lemma \ref{lm:separatrix} we have obtained the
parameterization $Z^0$ of the separatrix (see \eqref{def:separatrix}).
For the Hamiltonian $\HH_0$, we have a parameterization of the
separatrix $Z^\eps(t,\ga^0)=Z^0(t,\ga^0)+O(\eps)$.
The time origin can be chosen so that
\[
 Z^\eps(0,\ga^0)=\left(\frac{\pi}{2},
\Ga\sqrt{\frac{5}{3}}+O(\eps),\ga^0,\Ga\right),
\]
(recall that by Lemma \ref{lemma:IdentitiesOverSeparatrix}, $
Z^0(0,\ga^0)=(\pi/2,
\Ga\sqrt{5/3},\ga^0,\Ga)$).

In order to determine the splitting of separatrices, we consider the
section $g_1=\pi/2$ which is transversal to the flow locally in the
neigborhood ot the unperturbed separatrix. It is still transversal to
the perturbed stable and unstable invariant manifolds. We measure the
splitting in this section. The distance between of the invariant
manifolds is parameterized by $\ga^0$, the value of the $\gamma$
coordinate when $t=0$.

Melnikov Theory ensures that the transversal homoclinic points in the
section $g_1=\pi/2$ are $\eps^2=L_2^{-2}$-close to the non-degenerate
critical points of the Poincar{\'e}-Melnikov potential. Expanding the
potential in powers of $\eps$, we get
$\LL^\eps(\ga^0)=\LL(\ga^0)+\OO(\eps)$ with
\[
\begin{split}
  \LL(\ga^0)=&\int_{0}^{+\infty}
H_2\circ Z^0(t,\gamma^0)-H_2\circ Z_{\max}^0(t,\gamma^0)\,dt\\
&+\int_{-\infty}^{0}
H_2\circ Z^0(t,\gamma^0)-H_2 Z_{\min}^0(t,\gamma^0)\,dt.
\end{split}
\]
Moreover, since $H_2$ has a factor $e_1$ (see \eqref{eq:Foct}) it
vanishes over the periodic orbits $Z^0_{\min,\max}$. Thus, using
\eqref{def:PeriodicH0} and~\eqref{def:separatrix}, the potential reads
\begin{equation}\label{def:PoincareMelnikovPotential}
\begin{split}
  \LL(\ga^0) =&\int_{-\infty}^{+\infty}
  H_2(g_1^h(t),G_1^h(t),\ga^h(t),\Gamma)\,dt\\
  =&\int_{-\infty}^{+\infty} H_2(g_1^h(t),G_1^h(t),\gamma^0 +
  \gamma^1(t) + \gamma^2(t),\Gamma)\,dt.
\end{split}
\end{equation}
Melnikov theory then implies that, $\eps$-close to non-degenerate
critical points of $\LL$, there exist transversal homoclinic points in
the section $g_1=\pi/2$. The potential $\LL$ naturally depends on
parameters $L_1$ and $\Gamma$, which vary in the (non-empty) open set
\begin{equation}\label{def:Oset}
O = \left\{(L_1,\Ga): L_1>0, 0<\Gamma<L_1\sqrt{3/5}\right\} \subset
\RR^2.
\end{equation}
Let us write
$$\LL(\ga^0) = \LL_{L_1,\Gamma}(\ga^o).$$
The next proposition shows that $\LL$ necessarily has non-degenerate
critical points, except maybe for exceptional values of the
parameters.

\begin{proposition}
  \label{prop:Melnikov}%
  There exists a non constant real analytic function $\LL^+ : O \to
  \RR$ such that the potential~\eqref{def:PoincareMelnikovPotential}
  is of the form
  \[ \LL_{L_1,\Gamma} (\ga^0)=\LL^+_{L_1,\Gamma} e^{i\ga_0}+\ol
  {\LL^+}_{L_1,\Gamma} e^{-i\ga_0}.\] In particular, outside an
  analytic subset of $O$ of empty interior, $\LL^+$ does not vanish
  and thus $\LL$ (as a function of $\gamma^0$) has non-degenerate
  critical points.
\end{proposition}

This proposition is proven in Section \ref{sec:Melnikov}, where the
function $\LL^+$ is computed explicitly (see formula
\eqref{def:Lplus}). It is easier to describe the dynamics when the
dimension is as low as possible, so let us carry out the symplectic
reduction of the flow by time. Since $\dot\ga\neq 0$, define the
Poincar{\'e} section $\Sigma_{\ga_0}=\{\ga=\ga_0\}$ within some fixed
energy level, and the corresponding return map, induced by the
Hamiltonian \eqref{def:HSecExpanded},
\[
 \PP_{\ga_0}: \Sigma_{\ga_0} \longrightarrow \Sigma_{\ga_0}
\]
for which the Poincar{\'e} coordinates $(\xi,\eta)$ may be used (see
\eqref{def:PoincareVariables}). This map has a hyperbolic fixed point
$L_2^{-2}$-close to the origin with one dimensional stable and
unstable invariant manifolds.  The classical Melnikov theory applied
to the Melnikov potential obtained in Proposition \ref{prop:Melnikov}
entails the following corollary and theorem. Note that the circle
$g_1=\pi/2$ locally corresponds to the line $\xi=0$ (see
\eqref{def:PoincareVariables}).



\begin{corollary}
  \label{prop:Splitting}%
  Fix $(L_1^0,\Ga^0)\in O$ (defined in \eqref{def:Oset}) such that
  $\LL^+_{L_1^0,\Ga^0}\neq 0$ and let $\ga_0^*$ be a non-degenerate
  critical point of $\LL^+_{L_1^0,\Ga^0}$. For $L_2$ large enough,
  there exists some $\wt\ga_0^*=\ga_0^*+O(L_2^{-1})$ such that the
  Poincar{\'e} map $\PP_{\wt\ga_0^*}$ has a transversal homoclinic point
  in the line $\{\xi=0\}\subset \Sigma_{\wt\ga_0^*}$.
\end{corollary}
Recall that all Poincar{\'e} maps are conjugate. Therefore, this corollary 
implies that there are transversal homoclinic points for 
$\PP_{\ga_0}$ for all $\ga_0\in\TT$.

The transversality of invariant manifolds given by the corollary a
priori refers to the secular Hamiltonian obtained after one step of
averaging. Nevertheless, the conclusion holds for the Poincar{\'e} maps of
the secular Hamiltonian of any finite order. Indeed, higher order
averaging only modifies higher orders of the asymptotic expansion of
the splitting, while transversality at the first order was entailed by
the first order of the expansion, as given by the Poincar{\'e}-Melnikov
potential. Hence, by considering the analytic set defined by the
condition $\LL^+_{L_1^0,\Ga^0}\neq 0$ at all orders of averaging, one
gets the main result, where we restrict to some fixed compact set
$$K \subset O = \left\{(L_1,\Ga): L_1>0, 0<\Gamma<L_1\sqrt{3/5}\right\}
\subset \RR^2.$$

\begin{theorem}
  \label{thm:horseshoe}%
  Fix any $\ga_0\in\TT$. There is an analytic set $K^o$ of full measure in $K$ 
such that, for
  parameters $(L_1^0, \Ga^0)$ in $K^o$, for $L_2$ large enough, the
  Poincar{\'e} maps $\PP_{\ga_0}:\Sigma_{\ga_0}\mapsto\Sigma_{\ga_0}$ of
  the secular systems of all orders posess a horseshoe.
\end{theorem}

\subsection{Splitting of separatrices for the non-averaged 
Hamiltonian}\label{sec:Diffusion}

As explained in Remark \ref{rmk:HomoclinicLevelEnergy}, for the
secular Hamiltonian we have a hyperbolic periodic orbit at each energy
level within a compact interval of such levels. These periodic orbits
form a normally hyperbolic invariant
cylinder. Corollary~\ref{prop:Splitting} implies that the invariant
manifolds of this cylinder intersect transversally.

In Section \ref{sec:Melnikov} we consider $L_1$ and $L_2$ as
constants.  One can lift the dynamics to the extended phase space,
with the additional Keplerian coordinates $(\ell_1, L_1, \ell_2,
L_2)$. Then, the cylinder gets enlarged by four extra dimensions:
$(L_1, L_2)$, which are just constants of motion, and $(\ell_1,
\ell_2)$, which are performing a rigid rotation.

This extended system, after rescaling, is just the first order  of
\eqref{eq:F}, that is
\[
  F_0 = F_{\Kep} + F_{\secu}.
\]
Call $\La$ the cylinder of this Hamiltonian. Now one would like to
analyze Hamiltonian \eqref{eq:F}, that is, the full (non averaged)
three-body problem. The first step is to prove the persistence of the
invariant cylinder when $L_2$ is large enough, using that the reminder
is $O(L_2^{-k})$ for some $k\in\NN$.  Since we are in a singular
perturbation regime, the classical theory of persistence of normally
hyperbolic invariant manifolds \cite{Fen72} cannot be applied
directly.  However, now there are results which deal with singularly
perturbed problems and which can be applied to the present setting
\cite{Yang09,GuardiaLT15}. Note that we can do global averaging up to high 
order. Therefore, one does not need to face the problems in \cite{BernardKZ11} 
and the obtained cylinder can be as smooth as needed. Call $\wt \La$ the 
perturbed cylinder.

In this setting, one can expect Arnold diffusion i.e., orbits whose
action components drift by a large amount, uniformly with respect to
large $L_2$'s. Since
\begin{itemize}
\item $G_1$ is prescribed by $\wt\La$ and the homoclinic channel,
\item $G_2$ (resp. $L_1$) is constrained by the conservation of the
  angular momentum (resp. the energy),
\end{itemize}
one would expect to obtain orbits with drift in $L_2$, that is in the
semi major axis of the outer body. This is a remarkable feature, since
semi major axes are known to be very stable when two of the three
masses are very small \cite{Niederman96}.

To obtain unstable orbits one usually combines two types of
dynamics. The ``inner dynamics'' is the dynamics of Hamiltonian
\eqref{eq:F} in restriction to the cylinder $\wt\La$. The ``outer
dynamics'' and is the so-called scattering map \cite{DelshamsLS08},
obtained as the following limit. Assume that the invariant manifolds
of $\wt\La$ split transversally along a homoclinic chanel. Consider a
homoclinic orbit in the chanel, which is asymptotic to the trajectory
of some point $x_-\in\wt\Lambda$ as $t\rightarrow -\infty$ and to the
trajectory of some other point $x_+\in\wt\Lambda$ as $t\rightarrow
+\infty$. Then, we say that the scattering map $\SSS$ maps $x_-$ to
$x_+$. hat $\SSS$ be a map (as opposed to a more general
correspondance) is proved in~\cite{DelshamsLS08} under general
hypotheses which are satisfied here. Note that this map depends on the
chosen homoclinic chanel and therefore it may not be defined globally
--usually it is multivaluated. Provided that these two maps do not
have common invariant circles, by iterating them in a random order,
one gets pseudo-orbits which have the wanted unstable behavior. A
shadowing argument then permits to approximately realize these
pseudo-orbits as real orbits of the system.

Understanding both the inner and outer dynamics is certainly not easy
for Hamiltonian \eqref{eq:F}.  Concerning the inner dynamics, Jefferys
and Moser~\cite{Jefferys66} have used KAM theory to show that this
cylinder contains quasiperiodic hyperbolic tori which form a positive
measure set. There should be a very rich dynamics in the complement of
these tori in the cylinder. In particular, since the Hamiltonian
restricted to the cylinder has three degrees freedom, there may exist
Arnold diffusion in the cylinder itself.

Concerning the outer dynamics, one needs first to prove that the
invariant manifolds of the cylinder $\wt\La$ split transversally and
then derive some precise dynamical behavior for the scattering map, as
in~\cite{DelshamsLS08}. The analysis in the present paper allows us to
perform only the first step.

\begin{theorem}\label{thm:NonAveraged}
  Fix $L_1^+>L_1^->0$.  Consider the Hamiltonian \eqref{eq:F} with
  $L_1\in [L_1^-,L_1^+]$ and $L_2\geq L_2^0$. If $L_2^0$ is large
  enough,
  \begin{enumerate}
  \item there exists a normally hyperbolic invariant cylinder $\wt\La$,
  \item the invariant manifolds of the cylinder $\wt\La$ intersect
    transversally along a homoclinic chanel, which is diffeomorphic to
    $\wt\La$,
  \item and there exists a scattering map $\SSS:\wt\La\longrightarrow
    \wt\La$ associated to this homoclinic chanel.
  \end{enumerate}
\end{theorem}

Note that, in the above statement, the cylinder is invariant in the
sense that the vector field is tangent to the cylinder, but orbits may
escape from its boundary (sometimes one rather refers to these
manifolds as weakly invariant).

\begin{proof}
  As explained above, the persistence of the invariant cylinder can be
  obtained by the available results of persistence of normally
  hyperbolic invariant cylinders in the singular perturbation setting
  \cite{Yang09, GuardiaLT15}.

  For the other statements, we may use the results in
  Proposition~\ref{prop:Melnikov} and \cite{DelshamsLS08}. Using the
  expression of $F_\secu$ in \eqref{eq:Fav}, one can split the
  Hamiltonian \eqref{eq:F} as $F=H_0+H_1$ with $H_0=F_\Kep -\mu_1m_2
  F_{\quadr}$ and $H_1=F-H_0$.  Therefore,
  \[
  \begin{split} H_1=&-\mu_1 m_2 (\sigma_0-\sigma_1) F_{\oct} + O \left(
      \frac{a_1^4}{a_2^5} \right)\\ &-\mu_1 m_2 (\sigma_0-\sigma_1) F_{\oct}
    + O \left(L_2^{-5} \right).
  \end{split}
  \] and therefore satisfies $H_1=O(L_2^{-2})$.

  Let us abuse notation and assume that $H_0$ and $H_1$ are expressed
  in the variables given in \eqref{def:CoordinatesInGamma} and that
  time has been suitably scaled.  As we have explained the Hamiltonian
  $H_0$ has as normally hyperbolic invariant cylinder. Recall that
  $L_1$, $L_2$ and $\Ga$ are first integrals of $H_0$. Thus, the
  dynamics in this cylinder is foliated by three dimensional invariant
  tori. Let us define $Y_0\equiv Y_0(t, \ell_1^0, L_1^0, \ell_2^0,
  L_2^0, \ga^0, \Ga^0)$ the trajectory on the invariant tori
  $L=L_i^0$, and $\Ga=\Ga^0$ in the cylinder with initial condition in
  the $(\ell_1, \ell_2,\ga)$ variables given by $(\ell_1^0,
  \ell_2^0,\ga^0)$.

  Since $H_0$ is integrable, the stable and unstable invariant
  manifolds of the cylinder agree, and form a homoclinic manifold.
  The latter is given (at the first order) by the homoclinic of the
  quadrupolar Hamiltonian obtained in Lemma \ref{lm:separatrix} in the
  space $(g_1,G_1,\ga,\Ga)$.  Recall that for $H_0$ the variables
  $L_1$ and $L_2$ are first integrals. The dynamics of the variables
  $\ell_1$ and $\ell_2$ is close to a rigid rotation and can be easily
  deduced. These homoclinic manifold can be parameterized by time $t$
  and the coordinates of the cylinder $Y\equiv Y(t, \ell_1^0, L_1^0,
  \ell_2^0, L_2^0, \ga^0, \Ga^0)$. Note that $Y$ can be asymptotic to
  different points in the cylinder as $t\longrightarrow
  \pm\infty$. There exist smooth functions $(\ell^\pm_1,
  \ell^\pm_2,\ga^\pm)=(\ell^\pm_1, \ell^\pm_2,\ga^\pm)(\ell_1^0,
  L_1^0, \ell_2^0, L_2^0, \ga^0, \Ga^0)$ such that
\[
 Y(t, \ell_1^0, L_1^0, \ell_2^0, 
L_2^0, \gamma^0, \Gamma^0)- Y_0(t, \ell_1^\pm, L_1^0, \ell_2^\pm, L_2^0, 
\ga^\pm, \Ga^0)\longrightarrow 0
\]
as $t\longrightarrow\pm\infty$.

Then, to prove that the invariant manifolds split and that one can
define the scattering map of the full Hamiltonian $H_0+H_1$, one may
apply the results in \cite{DelshamsLS08} (see also
\cite{DelshamsLS06a}). Define the Poincar{\'e}-Melnikov potential
\[
\wt{ \mathcal L}( 
\ell_1^0, L_1^0, \ell_2^0, L_2^0, \gamma^0, 
\Gamma^0)=\int_{-\infty}^{0}H_1\circ Y -H_1\circ Y_0^-\,dt
+\int_{0}^{+\infty}H_1\circ Y -H_1\circ Y_0^+\,dt,
\]
where $Y$ stands for $Y(t, \ell_1^0, L_1^0, \ell_2^0, L_2^0, \gamma^0,
\Gamma^0)$ and $Y_0^\pm$ for $Y_0(t, \ell_1^\pm, L_1^0, \ell_2^\pm, 
L_2^0, \gamma^\pm, \Gamma^0)$. 

Consider the function
\begin{equation}\label{def:PoincareFunction}
 \tau\mapsto \wt{ \mathcal L}( 
\ell_1^0+\om_1\tau, L_1^0, \ell_2^0+\om_2\tau, L_2^0, \gamma^0+\om_3\tau, 
\Gamma^0)
\end{equation}
where $\om=(\om_1,\om_2,\om_3)$ is the frequency vector associated
with the torus $L=L_i^0$, and $\Ga=\Ga^0$. Results in
\cite{DelshamsLS08} imply that each non-degenerate critical point of
this function gives rise to a transversal intersection of the
invariant manifolds. The non-degeneracy of the critical point allows
us to define a local scattering map.

Using the formula of $H_1$, we see that
\[
 \wt{ \mathcal L}=\mathcal L+O \left(L_2^{-2} \right).
 \]
 where $\LL$ is the Melnikov potential introduced in
 \eqref{def:PoincareMelnikovPotential}. Proposition
 \ref{prop:Melnikov} implies that, as long as $\LL_{L_1^0,\Ga^0}^+\neq
 0$, the function $\LL(\ga^0-\om_1\tau)$ has non-degenerate critical
 points with respect to $\tau$. Thanks to the non-degeneracy, the
 function \eqref{def:PoincareFunction} has critical points, which are
 $O(L_2^{-2})$-close to those of $\LL(\ga^0-\om_1\tau)$. Each critical
 point of \eqref{def:PoincareFunction} gives rise to a transversal
 intersections and to an associated scattering map.
\end{proof}




Theorem \ref{thm:NonAveraged} implies that Hamiltonian \eqref{eq:F}
fits in the classical framework of Arnold diffusion along single
resonances: a normally hyperbolic invariant cylinder whose invariant
manifolds intersect transversally \cite{DelshamsLS00, DelshamsLS06a,
  BernardKZ11}.  Nevertheless, the results obtained in this paper are
not enough to obtain unstable orbits drifting along the
cylinder. Indeed, we cannot derive formulas for the scattering map and
therefore we have no information about the outer dynamics beyond the
fact that it is well defined.

The reason is the following. To prove that the invariant manifolds of
the cylinder split, it is enough to project them into a certain plane
and see that they intersect transversally in this plane. But, in order
to build unstable orbits, one needs more precise information on the
scattering map. Namely, to derive a first order approximation of the
scattering map, one needs to analyze how the invariant manifolds of
the cylinder split in every direction (see \cite{DelshamsLS08}), in
particular in the $L_1$ and $L_2$ directions. This would give the size
of the possible jumps that the scattering maps makes in these
directions. That is, how far can be two points in the cylinder which
are connected by a heteroclinic orbit. Since the mean anomalies
$\ell_1$ and $\ell_2$ are rapidly oscillating in our regime, the
transversality in the conjugate directions $L_1$ and $L_2$ is
exponentially small and therefore very difficult to detect.

  
An intermediate step would be to tackle the $1$-averaged Hamiltonian,
where only the fastest mean anomaly $\ell_1$ has been averaged
out. Then, one has a three degree of freedom Hamiltonian system with
only one fast frequency.  As a starter, we provide the expression of
the analogue to the quadrupolar term:
$$-\frac{\mu_1m_2}{\|q_2\|} \int_\TT \sigma_2 P_2(cos\zeta)
\frac{\|q_1\|^2}{\|q_2\|^3} \, d\ell_1 =$$
$$-\frac{3\mu_1m_2}{4\,a_1^2\, {r_2}^3} \left(
  \begin{array}[c]{l}
    \left(\left(4\,{e_1}^2+1\right)\,
      \sin ^2{g_1}+\left(1-\,{e_1}^2
      \right)\,\cos ^2{g_1}\right)\,\cos ^2i\,\sin ^2\left({v_2}+
      {g_2}\right)+\\
    10\,{e_1}^2\,\cos {g_1}\,\sin 
    {g_1}\,\cos i\,\cos \left({v_2}+{g_2}\right)\,\sin 
    \left({v_2}+{g_2}\right)+\\
    \left(\left(1-{e_1}^2\right)\,\sin ^2{g_1}+\left(4\,
        {e_1}^2+1\right)\,\cos ^2{g_1}
    \right)\,\cos ^2\left({v_2}+{g_2}\right)\\
    -{e_1}^2-\frac{2}{3}
  \end{array}
\right).$$
where $v_2\equiv v_2(G_2,L_2,\ell_2)$ is the eccentric anomaly.

Analyzing how the invariant manifolds split either for the full
three-body problem or for the $1$-averaged Hamiltonian, and deriving
formulas for the corresponding scattering maps, would be the major
step towards proving Arnold diffusion in the three body problem in the
lunar regime.

\section{Proof of Proposition \ref{prop:Melnikov}}
\label{sec:Melnikov}%

In Lemma \ref{lm:separatrix}, we have called \((g_1^h,G_1^h,\ga^h =
\gamma^0 + \gamma^1 + \gamma^2,\Gamma^h\equiv \Gamma)\) the separatrix
solution.\footnote{The variable $t$ used in section~\ref{sec:Fquad} is
  a rescaled time, since in Section~\ref{sec:Fquad} we have simplified
  the quadrupolar Hamiltonian by some multiplicative constant, when
  replacing $F_{\quadr}$ by $H_0$.} We compute the potential defined
in \eqref{def:PoincareMelnikovPotential}. To this end, we expand it in
Fourier series in $\gamma^0$. Since $H_2$ is a trigonometric
polynomial of degree $1$ in $\ga$ (or equivalently in $g_2$, see
equation~\eqref{eq:Foct}), it can be written as
\[
  H_2(g_1,G_1,\ga,\Gamma)=H_2^+
  (g_1,G_1,\Gamma)e^{i\ga}+H_2^- (g_1,G_1,\Gamma)e^{-i\ga}. 
\]
Because $\gamma^h = \gamma^0 + \gamma^1(t) + \gamma^2(t)$ depends
linearly on $\gamma^0$, $\LL(\ga^0)$ too has only two harmonics:
\begin{equation}
  \label{eq:L}%
  \LL(\ga_0)= \LL^+e^{i\ga^0}+\LL^-e^{-i\ga^0},
\end{equation}
where
\[
\LL^\pm= \int_{-\infty}^\infty H_2^\pm
(g_1^h(t),G_1^h(t),\Gamma)e^{\pm i (\gamma^1(t) + \gamma^2(t)) }dt.
\]
Since $\LL$ is a real function, it is determined by, say, the positive
harmonic $\LL^+ = \bar\LL^-$.

As a first step, we  parameterize the separatrix by $g_1$ instead
of $t$, using the following lemma (where we omit the upper index
$h$).

\begin{lemma}\label{lemma:IdentitiesOverSeparatrix}
  The following identites are satisfied on the separatrix given by
  Lemma~\ref{lm:separatrix}:
\[
\begin{cases}
  G_1 &=|\Gamma| \sqrt{\frac{5}{3}} \frac{\sin g_1}{\sqrt{1 -
      \frac{5}{3} \cos^2 g_1}}\\
  e_1 &= \sqrt{\frac{2}{3}} \frac{\Gamma}{L_1\chi} \sqrt{\frac{1 -
      \frac{5}{3} (1 + \chi^2) \cos^2 g_1}{1 -
      \frac{5}{3} \cos^2 g_1}}\\
  \cos i &= \sqrt{\frac{3}{5}} \frac{\sqrt{1- \frac{5}{3}
      \cos^2g_1}}{\sin g_1} \\
  \sin i &= \sqrt{\frac{2}{5}} \frac{1}{\sin g_1}\\
  \cos \gamma^2 &=\sqrt{1- \frac{5}{3} \cos^2 g_1}\\
  \sin \gamma^2 &=-\sqrt{\frac{5}{3}} \cos g_1.
\end{cases}
\]
\end{lemma}
The proof of this lemma is a direct consequence of Lemma \ref{lm:separatrix}
and formulas \eqref{def:eccentricity}, \eqref{eq:cosi}, \eqref{eq:sini}.
We can use this lemma to express the function
$\FF^+=H_2^{+}e^{i\ga_2}$ on the separatrix as a function of $g_1$.

\begin{lemma}\label{lemma:FFasg1function}
  The function $\FF^+$ can be written, on the separatrix, as a
  function of $g_1$ as $\FF^+=\frac{1}{2}(\FF_1+i\FF_2)$ with
  \[
  \begin{cases}
    \FF_1&= C_1 \frac{\sqrt{1-\frac{5}{3}\,\left(1+\chi^2\right) \,{\cos
          g_1}^2}}{1-\frac{5}{3}\,{\cos g_1}^2} \cos g_1 \left(1 -
      \frac{5}{3} \frac{1- \frac{11}{7}
        \frac{\Gamma^2}{L_1^2}}{1-\frac{5}{3}\frac{\Gamma^2}{L_1^2}}
      \cos
      g_1^2\right)\\
    \FF_2&=C_2\frac{\sqrt{1-\frac{5}{3}(1+\chi^2)\cos^2
        g_1}}{\sqrt{1-\frac{5}{3}\cos^2
        g_1}}\left(
      \begin{array}[c]{l}
        \frac{5}{3}\frac{\Ga^2}{L_1^2}\frac{\cos^2
          g_1(9-11\cos^2 g_1)}{1-\frac{5}{3}\cos^2
          g_1}+\\
        \frac{21}{5}-5\frac{\Ga^2}{L_1^
          2}-\left(14-11\frac{\Ga^2}{L_1^2}\right)\cos^2 g_1
      \end{array}
    \right)
  \end{cases}
  \]
  and
\[
C_1=\frac{105}{32}\frac{a_1^3}{a_2^4}\frac{e_2}{(1-e_2^2)^{5/2}}\left(1-\frac{5}
{3}\frac{\Ga^2}{L_1^2}\right)^{3/2} \quad \mbox{and} \quad 
C_2=\frac{15}{64}\sqrt{\frac{5}{3}}\frac{a_1^3}{a_2^4}\frac{e_2}{(1-e_2^2)^{5/2}
}\left(1-\frac{5}{3}\frac{\Ga^2}{L_1^2}\right)^{1/2}.
\]
\end{lemma}

\begin{proof}
We have 
\[
H_2=A_{\oct}e_1\left(\cos g_1\cos\ga \, A+\sin g_1\sin\ga \cos i
  \, B\right)
\]
where
\[
\begin{cases}
  A=\frac{G_1^2}{L_1^2}\left(5\sin^2 i(-7\cos^2 g_1 +6)-3\right)-35\sin^2 g_1 
\sin^2 i+7\\
  B=\frac{G_1^2}{L_1^2}\left(5\sin^2 i(7\cos^2 g_1 -4)+3\right)+35\sin^2 g_1 
\sin^2 i-7 
\end{cases}
\]
and 
\[A_{\oct}=- \frac{15}{64} \frac{a_1^3}{a_2^4} \frac{e_2}{(1-e_2^2)^{5/2}}.\]
Thus,
\[
\begin{split}
 \FF^+=\frac{A_{\oct}e_1}{2}\Big[&\left(A\cos g_1\cos \ga_2 +B\sin
   g_1\sin\ga_2\cos i \right)\\ 
 &+i \left(A\cos g_1\sin\ga_2 -B\sin g_1\cos\ga_2\cos i \right)\Big],
\end{split}
\]
which we want to express in terms of $g_1$. Using Lemma
\ref{lemma:IdentitiesOverSeparatrix}, the functions  $A$ and $B$ can be written
as
\[
\begin{cases}
 A=\frac{5}{3}\frac{\Ga^2}{L_1^2}\frac{1}{1-\frac{5}{3}\cos^2 g_1}(9-11\cos^2
g_1)-7\\
 B=-\frac{5}{3}\frac{\Ga^2}{L_1^2}\frac{1}{1-\frac{5}{3}\cos^2 g_1}(5-11\cos^2
g_1)+7.
\end{cases}
\]
Let
\[
\begin{cases}
 S_1=e_1(\cos g_1\cos \ga_2 A+\sin g_1\sin\ga_2\cos i B)\\
 S_2=e_1(\cos g_1\sin\ga_2 A-\sin g_1\cos\ga_2\cos i B).
\end{cases}
\]
Then,
\[
S_1=-\frac{14\Ga}{L_1\chi}\sqrt{\frac{2}{3}}\left(1-\frac{5}{3}\frac{\Ga^2}{
L_1^2}\right)
\frac{\sqrt{1-\frac{5}{3}(1+\chi^2)\cos^2 g_1}}{1-\frac{5}{3}\cos^2 g_1}\cos
g_1\left(1-\frac{5}{3}\frac{1-\frac{11}{7}
    \frac{\Ga^2}{L_1^2}}{1-\frac{5}{3}\frac{\Ga^2}{L_1^2}}\cos^2 g_1\right)
\]
and 
\[
\begin{split}
  S_2=&-\frac{\sqrt{10}}{3}\frac{\Ga}{L_1\chi}
  \frac{\sqrt{1-\frac{5}{3}(1+\chi^2)\cos^2
      g_1}}{\sqrt{1-\frac{5}{3}\cos^2 g_1}}\times \\
  & \left(\frac{5}{3}\frac{\Ga^2}{L_1^2} \frac{\cos^2 g_1(9-11\cos^2
      g_1)}{1-\frac{5}{3}\cos^2 g_1}+\frac{21}{5}-5\frac{\Ga^2}{L_1^
      2}-\left(14-11\frac{\Ga^2}{L_1^2}\right)\cos^2 g_1\right).
\end{split}
\]
\end{proof}

Now we compute $\FF^+$ as a function of $t$. Recall that the constant
$A_2$ was defined in \eqref{eq:A2-chi}.

\begin{lemma}\label{lemma:FFastaufunction}
  The functions $\FF_1$ and $\FF_2$ on the heteroclinic (cf. Lemma
  \ref{lemma:FFasg1function}), can be written as functions of $\tau =
  A_2 t$ as
\[
\begin{cases}
\FF_1=& \wt C_{1}\frac{\sinh\tau}{1+\sinh^2\tau}\cdot\frac{7+6\sinh^2\tau}{\chi^2+(1+\chi^2)\sinh^2\tau}\\
\FF_2=&C_2\frac{1}{\cosh \tau}\left[\frac{21}{5}-5\frac{\Ga^2}{L_1^2}+\right.\\
&\left.\frac{\sinh^2\tau}{\chi^2+(1+\chi^2)\sinh^2\tau}\left(\frac{\Ga^2}{L_1^2\chi^2}\frac{9\chi^2+\left(\frac{12}{5}+9\chi^2\right)\sinh^2\tau}{\cosh^2\tau}-\frac{3}{5}\left(14-11\frac{\Ga^2}{L_1^2}\right)\right)\right],
\end{cases}
\]
with
\[
\wt C_1=- \frac{15}{16\sqrt{10}}\frac{a_1^3}{a_2^4}\frac{e_2}{(1-e_2^2)^{5/2}}\frac{\Gamma}{L_1}\left(1-\frac{5}{3}\frac{\Gamma^2}{L_1^2}\right).
\]
\end{lemma}

\begin{proof}
We compute each term as a function of $\tau$. By \eqref{eq:cosg1-t}, we have that
\[
 \sqrt{1-\frac{5}{3}(1+\chi^2)\cos g_1^0(t)}=\frac{\chi}{\sqrt{\chi^2+(1+\chi^2)\sinh^2\tau}}
\]
and
\[
 1-\frac{5}{3}\cos g_1^0(t)=\chi^2\frac{1+\sinh^2\tau}{\chi^2+(1+\chi^2)\sinh^2\tau}.
\]
So, 
\[
\frac{ \sqrt{1-\frac{5}{3}(1+\chi^2)\cos g_1^0(t)}\cos g_1^0(t)}{ 1-\frac{5}{3}\cos g_1^0(t)}=\frac{1}{\chi}\sqrt{\frac{3}{5}}\frac{\sinh\tau}{1+\sinh^2\tau}.
\]
For the last term, we have
\[
 \left(1 -
    \frac{5}{3} \frac{1- \frac{11}{7}
      \frac{\Gamma^2}{L_1^2}}{1-\frac{5}{3}\frac{\Gamma^2}{L_1^2}} \cos^2
    g_1^0(t)\right)=\frac{\left(1-\frac{5}{3}\frac{\Gamma^2}{L_1^2}\right)\chi^2+M\sinh^2\tau}{\left(1-\frac{5}{3}\frac{\Gamma^2}{L_1^2}\right)\left(\chi^2+(1+\chi^2)\sinh^2\tau\right)}
\]
where 
\[
 M=\left(1-\frac{5}{3}\frac{\Gamma^2}{L_1^2}\right)(1+\chi^2)-\left(1-\frac{11}{7}\frac{\Gamma^2}{L_1^2}\right).
\]
Using the definition of $\chi$ in \eqref{eq:A2-chi}, we have that
\[
 \begin{split}
  \left(1-\frac{5}{3}\frac{\Gamma^2}{L_1^2}\right)\chi^2&=\frac{2}{3}\frac{\Gamma^2}{L_1^2}\\
 M&=\frac{4}{7}\frac{\Gamma^2}{L_1^2}.
 \end{split}
\]
Therefore, 
\[
 \left(1 -
    \frac{5}{3} \frac{1- \frac{11}{7}
      \frac{\Gamma^2}{L_1^2}}{1-\frac{5}{3}\frac{\Gamma^2}{L_1^2}} \cos^2  
g_1^0(t)\right)=\frac{2\Gamma^2}{21L_1^2\left(1-\frac{5}{3}\frac{\Gamma^2}{L_1^2
}\right)}\frac{7+6\sinh^2\tau}{\chi^2+(1+\chi^2)\sinh^2\tau}.
\]
Putting all these formulas together and defining
\[
 \wt C_1=\frac{2}{21\chi}C_1\sqrt{\frac{3}{5}}\frac{\Gamma^2}{L_1^2}\frac{1}{1-\frac{5}{3}\frac{\Gamma^2}{L_1^2}}.
\]
we complete the proof.
\end{proof}
The formulas of Lemma \ref{lemma:FFastaufunction} allow us to compute  the
Poincar{\'e} - Melnikov potential \eqref{eq:L}. We split it as 
as $\LL^+=\LL_1+i\LL_2$ with
\begin{equation}\label{def:Lj}
 \LL_j=\frac{1}{2}\int_\RR \FF_j(t)e^{i\ga^1(t)}\,dt,\,\,\,j=1,2.
\end{equation}
We first compute 
\begin{eqnarray*}
  \LL_1&= &\wt C_1 \int_{-\infty}^\infty
\frac{\sinh(A_2t)}{1+\sinh^2(A_2t)}\cdot\frac{7+6\sinh^2(A_2t)}{
\chi^2+(1+\chi^2)\sinh^2(A_2t)}e^{- i\frac{2\Gamma}{L_1^2}t }dt\\
  &=&\frac{\wt C_1}{A_2}\int_{-\infty}^\infty
\frac{\sinh\tau}{1+\sinh^2\tau}\cdot\frac{7+6\sinh^2\tau}{
\chi^2+(1+\chi^2)\sinh^2\tau}e^{- i\frac{2\Gamma}{A_2L_1^2}\tau }d\tau.
\end{eqnarray*}
using the residue theorem. First note that if we look at this
integral along the path $\tau=-i\pi+\sigma$, $\sigma\in\RR$, instead
of the real line, we get
\[
\begin{split}
\int_{-i\pi+\RR}&
\frac{\sinh\tau}{1+\sinh^2\tau}\cdot\frac{7+6\sinh^2\tau}{
\chi^2+(1+\chi^2)\sinh^2\tau}e^{- i\frac{2\Gamma}{A_2L_1^2}\tau }d\tau\\
&=-e^{-\frac{2\Gamma\pi}{A_2L_1^2}}\int_{\RR}
\frac{\sinh\sigma}{1+\sinh^2\sigma}\cdot\frac{7+6\sinh^2\sigma}{
\chi^2+(1+\chi^2)\sinh^2\sigma}e^{- i\frac{2\Gamma}{A_2L_1^2}\sigma }d\sigma. 
\end{split}
\]
So, if we define  the function 
\begin{equation}\label{def:IntegrandReal}
f_1(\tau)=\frac{\sinh\tau}{1+\sinh^2\tau}\cdot\frac{7+6\sinh^2\tau}{
\chi^2+(1+\chi^2)\sinh^2\tau}e^{- i\frac{2\Gamma}{A_2L_1^2}\tau },
\end{equation}
Cauchy's integral formula shows
\begin{equation}
  \label{eq:L1}
  \LL_1 =-\frac{2\pi i}{1+e^{-\frac{2\Gamma\pi}{A_2L_1^2}}}\sum
\mathrm{Residues}
\end{equation}
where the summation stands over the residues of the function $f_1$ in
$-\pi<\Im \tau<0$; the negative sign comes from the index of the
curve.

\begin{lemma}\label{lemma:ResiduesReal}
  The function $f_1$ defined in \eqref{def:IntegrandReal} has three
  singularities in the region $-\pi<\Im \tau<0$ given by
  \[
  \begin{cases}
    a_0^-=-i\frac{\pi}{2}\\
    a_1^-=-i\arcsin\frac{\chi}{\sqrt{1+\chi^2}}\\
    a_2^-=-i\left(\pi+\arcsin\frac{\chi}{\sqrt{1+\chi^2}}\right)
  \end{cases}
  \]
  and the associated residues are 
  \[
  \begin{cases}
    \Res\left(f, a_0^-\right)=-\frac{2\Gamma}{A_2^2}e^{-
\frac{\pi\Gamma}{A_2L_1^2} }\\
    \Res\left(f, a_1^-\right)=\frac{7+\chi^2}{2\sqrt{1+\chi^2}}e^{-
\frac{2\Gamma}{A_2L_1^2} \arcsin\frac{\chi}{\sqrt{1+\chi^2}}}\\
    \Res\left(f, a_2^-\right)=-\frac{7+\chi^2}{2\sqrt{1+\chi^2}}e^{-
\frac{2\Gamma}{A_2L_1^2}\left(\pi-\arcsin\frac{\chi}{\sqrt{1+\chi^2}}\right) }.
  \end{cases}
  \]
\end{lemma}

Before proving the lemma, we proceed to deduce the value of the real
part of the Melnikov integral. The lemma shows that the sum of
formula~\eqref{eq:L1} is
\[
\sum \mathrm{Residues} = 
-\frac{2\Gamma}{A_2L_1^2}e^{- \frac{\pi\Gamma}{A_2L_1^2} }
+\frac{7+\chi^2}{2\sqrt{1+\chi^2}} e^{-
  \frac{2\Gamma}{A_2L_1^2} \arcsin\frac{\chi}{\sqrt{1+\chi^2}}}\left( 1 - e^{-
    \frac{2\Gamma}{A_2L_1^2} \pi} \right),
\]
hence the function $\LL_1$ introduced in \eqref{def:Lj} satisfies
\[
\LL_1=-\frac{2\pi i}{1+e^{-\frac{2\Gamma\pi}{A_2L_1^2}}} \frac{\wt
  C_1}{A_2}
\left(
  \begin{array}[c]{l}
    -\frac{2\Gamma}{A_2L_1^2}e^{- \frac{\pi\Gamma}{A_2L_1^2} }+\\
    +\frac{7+\chi^2}{2\sqrt{1+\chi^2}} e^{-
      \frac{2\Gamma}{A_2L_1^2} \arcsin\frac{\chi}{\sqrt{1+\chi^2}}}\left( 1 -
e^{-
        \frac{2\Gamma}{A_2L_1^2} \pi} \right)
    \end{array}
  \right). 
\]

\begin{proof}[Proof of Lemma \ref{lemma:ResiduesReal}]
The singularities are the solutions of
\[
 1+\sinh^2\tau=0\,\,\,\text{or either }\,\,\,\chi^2+(1+\chi^2)\sinh^2\tau=0.
\]
The first equation is just    $1+\sinh^2\tau=\cosh^2 \tau=0$ and the only possible solution in $-\pi<\Im \tau<0$ is $ a_0^-=-i\pi/2$. For the second equation, one can take $\tau=i\sigma$. Then, it is equivalent to 
\[
 \sin\sigma=\pm\frac{\chi}{\sqrt{1+\chi^2}},
\]
which gives the other two singularities. 

Now we compute the residues. We start by $a_0^-$. We compute the Laurent series of each term in $f$. From the fact that
 \[\sinh\tau=-i \left(1+\OO_2(\tau+i\pi/2)\right)\]
one can easily deduce  the following,
\[
 \begin{split}
  \frac{1}{\cosh^2\tau}&=-\frac{1}{(\tau+i\pi/2)^2}\left(1+\OO_2(\tau+i\pi/2)\right)\\
   7+6\sinh^2\tau&=1+\OO_2(\tau+i\pi/2)\\
\chi^2+(1+\chi^2)\sinh^2\tau&=-1+\OO_2(\tau+i\pi/2).
 \end{split}
\]
Therefore
\[
 \frac{\sinh\tau}{1+\sinh^2\tau}\cdot\frac{7+6\sinh^2\tau}{
\chi^2+(1+\chi^2)\sinh^2\tau}=-\frac{i}{(\tau+i\pi/2)^2}
\left(1+\OO_2(\tau+i\pi/2)\right).
\]
For the exponential, we know have that
\[
 e^{- i\frac{2\Gamma}{A_2L_1^2}\tau }=e^{- \frac{\pi\Gamma}{A_2L_1^2}}-
i\frac{2\Gamma}{A_2L_1^2}e^{- \frac{\pi\Gamma}{A_2L_1^2}
}(\tau+i\pi/2)+\OO_2(\tau+i\pi/2).
\]
From these two last expansions, we obtain the residue for $a_0^-=-i\pi/2$.

Now we compute the other two residues. We compute them at the same time. Note that for $i=1,2$,
\[
\sinh a_i^-=-i\frac{\chi}{\sqrt{1+\chi^2}}
\]
and 
\[\cosh a_1^-=\frac{1}{\sqrt{1+\chi^2}}, \quad \cosh
a_2^-=-\frac{1}{\sqrt{1+\chi^2}}.\]
Therefore
\[
 \begin{split}
   7+6\sinh^2\tau&=\frac{7+\chi^2}{1+\chi^2}+\OO_1(\tau-i\pi/2)\\
\chi^2+(1+\chi^2)\sinh^2\tau&=-2\chi i(\tau-a_1^-)+\OO_2(\tau-a_1^-)\\
\chi^2+(1+\chi^2)\sinh^2\tau&=+2\chi i(\tau-a_2^-)+\OO_2(\tau-a_2^-)
 \end{split}
\]
and we have also 
\[
\begin{split}
 e^{- i\frac{2\Gamma}{A_2L_1^2}\tau }&= e^{- \frac{2\Gamma}{A_2L_1^2}
\arcsin\frac{\chi}{\sqrt{1+\chi^2}}}+\OO_1(\tau-a_1^-)\\
 e^{- i\frac{2\Gamma}{A_2L_1^2}\tau }&= e^{-
\frac{2\Gamma}{A_2L_1^2}\left(\pi-\arcsin\frac{\chi}{\sqrt{1+\chi^2}}\right)
}+\OO_1(\tau-a_2^-).
\end{split}
\]
With all these expansions, one can easily deduce the last two residues.
\end{proof}

Analogously, the function $\LL_2$ introduced in \eqref{def:Lj}, satisfies
\[
 \LL_2=-\frac{C_2}{A_2} \frac{2\pi i}{1+e^{-\frac{2\Gamma\pi}{A_2L_1^2}}}
 \sum \mathrm{Residues}, 
\]
where the sum stands over all residues in $-\pi<\tau<0$ of
\[
\begin{split}
 f_2(\tau)=&\frac{e^{- i\frac{2\Gamma}{A_2L_1^2}\tau }}{\cosh
   \tau}\left[\frac{21}{5}-5\frac{\Ga^2}{L_1^2}+\right.\\ 
&\left.\frac{\sinh^2\tau}{\chi^2+(1+\chi^2)\sinh^2\tau}\left(\frac{\Ga^2}{
L_1^2\chi^2}
\frac{9\chi^2+\left(\frac{12}{5}+9\chi^2\right)\sinh^2\tau}{\cosh^2\tau}-\frac{3
}{5}\left(14-11\frac{\Ga^2}{L_1^2}\right)\right)\right].
\end{split}
\]

\begin{lemma}
  \label{lemma:ResiduesImaginary:2}%
   In the   region $-\pi<\Im \tau<0$, the function $f_2$ has the same singularities as $f_1$, given in Lemma \ref{lemma:ResiduesReal}. Moreover, the associated residues are given by
  \[
  \begin{cases}
    \Res\left(f_2,
a_0^-\right)=-\frac{i}{5}\left(21-8\frac{\Ga^2}{L_1^2}+24\frac{\Ga^4}{
A_2^2L_1^6\chi^2}\right)e^{-\frac{\pi\Gamma}{A_2L_1^2}}\\
    \Res\left(f_2, a_1^-\right)=
-i\frac{3}{5}\frac{\chi}{(1+\chi^2)^{3/2}}\left(11\frac{\Ga^2}{L_1^2}-7\right)
e^{-\frac{2\Gamma}{A_2^2} \arcsin\frac{\chi}{\sqrt{1+\chi^2}}}\\
    \Res\left(f_2, a_2^-\right)= 
-i\frac{3}{5}\frac{\chi}{(1+\chi^2)^{3/2}}\left(11\frac{\Ga^2}{L_1^2}-7\right)
e^{-\frac{2\Gamma}{A_2^2} \left(\pi-\arcsin\frac{\chi}{\sqrt{1+\chi^2}}\right)}
  \end{cases}
  \]
\end{lemma}
Hence, the function $\LL_2$ introduced in \eqref{def:Lj} satisfies
\[
 \LL_2=-\frac{C_2}{A_2} \frac{2\pi i}{1+e^{-\frac{2\Gamma\pi}{A_2L_1^2}}}
 \left(
   \begin{array}[c]{l}
     -\frac{i}{5}\left(21-8\frac{\Ga^2}{L_1^2}+
       24\frac{\Ga^4}{A_2^2L_1^6\chi^2}\right)e^{-\frac{\pi\Gamma}{A_2^2}}
     + \\  
     -i\frac{3}{5}\frac{\chi}{(1+\chi^2)^{3/2}}
     \left(11\frac{\Ga^2}{L_1^2}-7\right) e^{-\frac{2\Gamma}{A_2^2}
       \arcsin\frac{\chi}{\sqrt{1+\chi^2}}} + \\ 
     -i\frac{3}{5}\frac{\chi}{(1+\chi^2)^{3/2}}
     \left(11\frac{\Ga^2}{L_1^2}-7\right) e^{-\frac{2\Gamma}{A_2^2}
       \left(\pi-\arcsin\frac{\chi}{\sqrt{1+\chi^2}}\right)}
   \end{array}
 \right).
\]

\begin{proof}[Proof of Lemma \ref{lemma:ResiduesImaginary:2}]
  We use the expansions obtained in the proof of Lemma
  \ref{lemma:ResiduesReal}.  For the last two residues, it is
  enough to use also that
\[
\frac{\Ga^2}{L_1^2\chi^2}\frac{9\chi^2+\left(\frac{12}{5}+9\chi^2\right)\sinh^2\tau}{\cosh^2\tau}=\frac{33}{5}\frac{\Ga^2}{L_1^2}+\OO\left(\tau-a_i^-\right)
\]
for $i=1,2$. Then, using the expansions  of Lemma \ref{lemma:ResiduesReal}, we can deduce the two last residues.

Now we compute the residue at $\tau=a_0$. We split $f_2$ into three parts $f_2=h_1+h_2+h_3$ with 
\[
\begin{cases}
 h_1(\tau)=& \left(\frac{21}{5}-5\frac{\Ga^2}{L_1^2}\right)\frac{e^{-
i\frac{2\Gamma}{A_2L_1^2}\tau }}{\cosh
   \tau} \\
h_2(\tau)=&-\frac{3}{5}\left(14-11\frac{\Ga^2}{L_1^2}\right)\frac{\sinh^2\tau}{
\chi^2+(1+\chi^2)\sinh^2\tau}\frac{e^{- i\frac{2\Gamma}{A_2L_1^2}\tau }}{\cosh
   \tau}\\
h_3(\tau)=&\frac{\Ga^2}{L_1^2\chi^2}\frac{9\chi^2+\left(\frac{12}{5}
+9\chi^2\right)\sinh^2\tau}{\chi^2+(1+\chi^2)\sinh^2\tau}\sinh^2\tau\frac{e^{-
i\frac{2\Gamma}{A_2L_1^2}\tau }}{\cosh^ 3
   \tau}.
\end{cases}
\]
Functions $h_1$ and $h_2$ have a pole of order one at $\tau=a_0^-$ so
one can proceed as for the others singularies. However, $h_3$ has a
pole of order 3 and therefore, we need to compute the expansion up to order 3
of each term in $h_2$.

We use the expansions
\[
 \begin{cases}
   \sinh\tau&=-i-\frac{i}{2}y^2+\OO_4(y)\\
   \cosh\tau&=-iy-\frac{i}{6}y^3+\OO_5(y)
 \end{cases}
\]
where $y=\tau+i\frac{\pi}{2}$. We use this notation throughout the proof.

We start with $h_1$ and $h_2$. It can be easily seen that 
\[
 \begin{cases}
  \Res\left(h_1,
a_0^-\right)&=i\left(\frac{21}{5}-5\frac{\Ga^2}{L_1^2}\right)e^{-\frac{\pi\Gamma
}{A_2L_1^2}}\\
  \Res\left(h_2,
a_0^-\right)&=-i\frac{3}{5}\left(14-11\frac{\Ga^2}{L_1^2}\right)e^{-\frac{
\pi\Gamma}{A_2L_1^2}}
 \end{cases}
\]

For $h_3$ we use the following expasions
\[
\begin{split}
\frac{9\chi^2+\left(\frac{12}{5}+9\chi^2\right)\sinh^2\tau}{\chi^2+(1+\chi^2)\sinh^2\tau}\sinh^2\tau&=-\frac{12}{5}-\left(\frac{12}{5}+\frac{33}{5}\chi^2\right)y^2+\OO_4(y)\\
 \frac{1}{\cosh^3\tau}&=-\frac{i}{y^3}\left(1-\frac{y^2}{2}+\OO_4(y)\right)\\
 e^{- i\frac{2\Gamma}{A_2L_1^2}\tau
}&=e^{-\frac{\pi\Gamma}{A_2L_1^2}}\left(1-i\frac{2\Ga}{A_2L_1^2}y-\frac{2\Ga^2}{
A_2^2L_1^4}y^2+\OO_4(y)\right).
\end{split}
 \]
Putting together all these expansions one can deduce the residue
\[
   \Res\left(h_3,
a_0^-\right)=i\frac{\Ga^2}{L_1^2\chi^2}\left(\frac{33}{5}\chi^2-\frac{24}{5}
\frac{\Ga^2}{A_2^2L_1^4}\right)e^{-\frac{\pi\Gamma}{A_2L_1^2}}
\]
Now it only remains to add the three residues to obtain the residue of $f_2$ at $\tau=a_0^-$.
\end{proof}

Gathering what precedes, 
and letting
$$\alpha = \frac{\pi\Gamma}{A_2L_1^2}, \quad \hat\Gamma =
\frac{\Gamma}{L_1}, \quad \hat\chi = \frac{\chi}{\sqrt{1+\chi^2}}
\quad \mbox{and} \quad \beta = \frac{\alpha}{\pi} \arcsin\hat\chi,$$ we
obtain the following analytic expression:
\begin{equation}\label{def:Lplus}
\LL^+ = -\frac{2\pi i}{A_2 \left(1+e^{-2\alpha}\right)} \left(
  \begin{array}[c]{l}
    \wt C_1
\left(
  \begin{array}[c]{l}
    -\frac{2\alpha}{\pi} e^{-\alpha}+\\
    +\frac{7+\chi^2}{2\chi}\sin\frac{\beta\pi}{\alpha} e^{-2\beta}
    \left( 1 - e^{- 2\alpha} \right)
  \end{array}
\right) \\
+  C_2 
\left(
  \begin{array}[c]{l}
    \frac{1}{5}\left(21-8\hat\Ga^2+
      24 \frac{\hat\Gamma^2 \alpha^2}{\pi^2\chi^2}\right)e^{-\alpha} +
    \\   
    +\frac{3}{5}\frac{\chi}{(1+\chi^2)^{3/2}}
    \left(11\hat\Ga^2-7\right) e^{-2\beta} + \\ 
    +\frac{3}{5}\frac{\chi}{(1+\chi^2)^{3/2}}
    \left(11\hat\Ga^2-7\right) e^{-2(\alpha-\beta)} 
  \end{array}
\right)
\end{array}
\right).
\end{equation}
In order to check that $\LL^+$ is not consant, notice that,
asymptotically when $\Gamma$ tends to $0$ (and $L_1$ is kept
constant), $\alpha = O(\Gamma^{-2})$ while $\beta = O(\Gamma^{-1})$,
so that the term in $\tilde C_1\, e^{-2\beta}$ dominates the others:
for some $C_3(L_1)$ and $C_4(L_1)$ independant of $\Gamma$,
$$\LL^+ \sim \frac{C_3(L_1)}{\Gamma^2} \exp \left(-2 C_4(L_1)\Gamma +
  O(\Gamma^{-2})\right).$$ Since the dominant term of the right hand
side is a non constant function of $\Gamma$,
Proposition~\ref{prop:Melnikov} is proved.

\bigskip \emph{Funding: } J. F. is partially supported by the
French ANR (projet ANR-10-BLAN 0102 DynPDE). M. G. is partially
supported by the Spanish MINECO-FEDER Grant MTM2012-31714 and the
Catalan Grant 2014SGR504.

\emph{Conflict of Interest: } The authors declare that they have no
conflict of interest.

\bibliography{references} \bibliographystyle{alpha}

\end{document}